\documentclass[12pt]{article}
\usepackage{amsmath,latexsym,amssymb,amsthm,enumerate,amsmath,amscd}
\usepackage{amsthm,amsmath,amsfonts,amssymb,amscd,latexsym,stmaryrd}
\usepackage[mathscr]{eucal}
\usepackage{color}
\usepackage[usenames,dvipsnames]{xcolor}
\usepackage{colortbl}
\usepackage[all,cmtip]{xy}
\usepackage{graphicx}
\usepackage{comment}
\usepackage{hyperref}
\usepackage{tikz}
\usetikzlibrary{calc,shadings,patterns}
\usepackage{epsfig}
\usepackage[enableskew,vcentermath]{youngtab}
\theoremstyle{plain}
\newtheorem*{utheorem}{Theorem}
\newtheorem*{ucorollary}{Corollary}
\newtheorem{theorem}{Theorem}[section]
\newtheorem{proposition}[theorem]{Proposition}

\newtheorem{lemma}[theorem]{Lemma}
\theoremstyle{definition}
\newtheorem{definition}[theorem]{Definition}

\newtheorem{remark}[theorem]{Remark}
\newtheorem{remarks}[theorem]{Remarks}
\newtheorem{example}[theorem]{Example}

\newtheorem{question}[theorem]{Question}

\newtheorem{thm}{Theorem}

\newtheorem{rem}[thm]{Remark}

\newtheorem*{remark*}{Remark}

\newtheorem*{ack}{Acknowledgment}

\numberwithin{equation}{section}
\numberwithin{table}{section}
\definecolor{purple}{rgb}{0.4,0.2,0.4}

\def\cha{\mathrm{char}\ }

\DeclareMathOperator{\Ann}{Ann}

\def\Hilb{\mathrm{Hilb}}

\def\Soc{\mathrm{Soc}}

\def\coloneqq{:=}
\def\<{\left<}
\def\>{\right>}
\def\Gl{\mathrm{Gl}}
\def\A{\mathcal{A}}

\def\C{\mathbb{C}}

\def\F{{\sf k}}

\def\Sym{\mathrm{Sym}}
\def\Spec{\mathrm{Spec}}

\def\ns{\footnotesize \it}

\def\Z{\mathbb{Z}}

\newcommand{\maxA}{\ensuremath{{\mathfrak{m}_{\mathcal A}}}}

\definecolor{med-gray}{gray}{0.5}
\definecolor{gray1}{gray}{0.87}
\definecolor{gray2}{gray}{0.74}
\definecolor{gray3}{gray}{0.64}
\definecolor{gray4}{gray}{0.48}
\definecolor{verylight-yellow}{rgb}{1,1,0.7}
\definecolor{yellow}{rgb}{1,1,0.2}
\definecolor{vivid-blue}{rgb}{0.2,0,1}
\definecolor{light-pink}{rgb}{1,0.8,1}
\definecolor{med-pink}{rgb}{1,0.6,1}
\definecolor{aqua}{rgb}{0.0, 1.0, 1.0}
\definecolor{light-gray}{rgb}{0.5, 0.9, 0.5}

\def\cha{\mathrm{char}\ }

\title{Free extensions and Jordan type}

\author{Anthony Iarrobino\\[.05in]
{\ns Department of Mathematics, Northeastern University, Boston, MA 02115,
 USA.
}\\[.2in] Pedro Macias Marques\\[.05in]
{\ns Departamento de Matem\'{a}tica, Escola de Ci\^{e}ncias e Tecnologia, Centro de Investiga\c{c}\~{a}o}\\[-.05in]
{\ns  em Matem\'{a}tica e Aplica\c{c}\~{o}es, Instituto de Investiga\c{c}\~{a}o e Forma\c{c}\~{a}o Avan\c{c}ada,}\\[-.05in]
{\ns Universidade de \'{E}vora, Rua Rom\~{a}o Ramalho, 59, P--7000--671 \'{E}vora, Portugal.}
\\[.2in]
Chris McDaniel\\[0.05in]
{\ns Department of Mathematics, Endicott College, Beverly, MA 01915, USA.}
}

\date{May 21, 2019, revised October 23, 2019}
\begin{document}
\maketitle
\begin{abstract}
Free extensions of graded Artinian algebras were introduced by T.~Harima and J.~Watanabe, and were shown to preserve the strong Lefschetz property. The Jordan type of a multiplication map $m$ by a nilpotent element of an Artinian algebra is the partition determining the sizes of the blocks in a Jordan matrix for $m$. We show that a free extension $C$ of the Artinian algebra $A$ with fiber $B$  is a deformation of the usual tensor product. This has consequences for the generic Jordan types  of $A,B$ and $C$: we show that the Jordan type of $C$ is at least that of the usual tensor product in the dominance order (Theorem~\ref{TPthm}).  In particular this gives a different proof of the T.~Harima and J.~Watanabe result concerning the strong Lefschetz property of a free extension. Examples illustrate that a non-strong-Lefschetz graded Gorenstein algebra $A$ with non-unimodal Hilbert function may nevertheless have a non-homogeneous element with strong Lefschetz Jordan type, and may have an $A$-free extension that is strong Lefschetz.\par
 We apply these results to algebras of relative coinvariants of linear group actions on a polynomial ring.
 \end{abstract}
\section{Introduction} 
Let $\mathcal{A}=(\mathcal{A},\mathfrak{m},\sf k)$ be a local Artinian algebra over a field ${\sf k}=\mathcal A/\mathfrak {m}$ where $\mathfrak m$ is the unique maximum ideal.  The Jordan type of an element $\ell\in\mathfrak{m}$ is the partition $P_\ell$ with parts equal to the sizes of the Jordan blocks of its associated (nilpotent) multiplication map $\times\ell\colon \mathcal{A}\rightarrow \mathcal{A}$.  The generic Jordan type of $\mathcal{A}$ is defined to be the largest partition $P_{\mathcal{A}}=P_\ell$, with respect to the dominance order on partitions, which occurs among the elements $\ell\in\mathfrak{m}$.  It is known that the Jordan types $P_\ell$ are bounded above by the conjugate partitions $H(\mathcal{A})^\vee$ of their associated Hilbert functions $H(\A)$.  We say that an element $\ell\in\mathfrak{m}$ has strong Lefschetz Jordan type (SLJT) if its Jordan type achieves this bound, i.e. if $P_\ell=H(\mathcal{A})^\vee$.  We say that the algebra $\mathcal{A}$ has strong Lefschetz Jordan type if it admits a SLJT element.\par
 Our graded algebras $A$ (we use roman letters to distinguish this case) may or may not be \emph{standard graded}:  that is, generated by $A_1$ over $A_0=\sf k$.
When $A$ is standard graded one typically restricts to linear forms $\ell\in A_1$, and computes the generic linear Jordan type of $A$, $P_A$, as the largest Jordan type which occurs among the linear forms $\ell\in A_1$.  We say that a linear form $\ell\in A_1$ is strong Lefschetz (SL) if $P_\ell=H(A)^\vee$, the conjugate partition of the Hilbert function of $A$ corresponding to the grading on $A$. This Hilbert function $H(A)$ may in general be different from the Hilbert function $H(\mathcal{A})$, with respect to the filtration by powers of the maximal ideal $\mathfrak m_\A=\oplus_{k=1}^\infty A_k$, where $\mathcal A$ denotes the related local algebra. The strong Lefschetz property of a linear form $\ell\in A_1$ is equivalent to the familiar condition that the multiplication maps $\times\ell^k\colon A_i\rightarrow A_{i+k}$ have maximal rank for all integers $i,k$.  We say that the graded algebra $A$ is strong Lefschetz (SL) if it admits a strong Lefschetz linear element.  \par
Let $A$, $B$, and $C$ be graded Artinian $\sf k$-algebras (not necessarily standard-graded).  We say that $C$ is a free extension of $A$ with fiber $B$ if there is a map of graded algebras $\iota\colon A\rightarrow C$ making $C$ into a free $A$-module and there is an isomorphism of graded algebras $B\cong C/(\iota(\mathfrak{m}_A))$.  The tensor product algebra $A\otimes_{\sf k}B$ is a free extension of $A$ with fiber $B$, and one of our principal results is that a free extension may be regarded as a deformation of the tensor product (Theorem \ref{thm:def}).

We use this to show a main goal of this note, that the generic linear Jordan type of a free extension is bounded below by that of its associated tensor product algebra: 

\begin{utheorem}(Theorem \ref{TPthm} below)
	Let $A$, $B$, and $C$ be graded Artinian algebras over an infinite field $\sf k$, and suppose that $C$ is a free extension of $A$ with fiber $B$.  Let $P_{A\otimes_{\sf k} B}$ be the generic linear Jordan type of the tensor product algebra, and let $P_C$ be the generic linear Jordan type of the $A$-free extension $C$.  Then in the dominance partial order we have 
	$$P_C\geq P_{A\otimes_{\sf k} B}.$$
\end{utheorem}

This theorem can be viewed as a generalization of a well known result of T. Harima and J. Watanabe, which we derive as a corollary. For a graded algebra $j_A$ is the largest degree $i$ for which $A_i\not=0$.
\begin{ucorollary}(Theorem \ref{SL2thm} below)\cite[Theorem 6.1]{HW}
	Let the Artinian algebra $C$ be a free extension of $A$ with fiber $B$. Assume that $\cha {\sf k}=0$ or $\cha {\sf k}> j_A+j_B$, that the Hilbert functions of both $A$ and $B$ are symmetric, and that both $A$ and $B$ are strong Lefschetz. Then $C$ is also strong Lefschetz. 
\end{ucorollary}

The invariant theory of finite groups abounds with free extensions, and the above corollary can be used to show that coinvariant algebras associated to certain finite reflection groups are SL.  On the other hand, we also find examples of free extensions $C$ over $A$ with fiber $B$ where the inequality in the theorem is strict, i.e. $P_C>P_{A\otimes_{\sf k}B}$.  Example \ref{SL2counterex} shows that this strict inequality can even occur if the gradings on $A$, $B$, and $C$ are all standard, answering a question of J. Watanabe (private communication).  For further study of relative coinvariant rings from this viewpoint see \cite{McDCI}.

\subsection{Jordan Type and Strong Lefschetz}
Let $\mathcal{A}=\left(\mathcal{A},\mathfrak{m},\sf k\right)$ be a local Artinian algebra over a field $\sf k$.  Then for any element $\ell\in\mathfrak{m}$ the multiplication map \begin{equation}
\label{eq:MultMap}
\times\ell\colon \mathcal{A}\rightarrow \mathcal{A}
\end{equation} 
is a nilpotent linear transformation of a finite dimensional vector space $\mathcal{A}$.  Define the \emph{Jordan type of $\ell$} to be the partition $P_\ell$ with parts equal to the block sizes of the Jordan canonical form of the multiplication map \eqref{eq:MultMap}. 

We recall the \emph{dominance order} on partitions.
	Let $P=(p_1,\ldots ,p_s)$ and $ P'=(p'_1,\ldots, p'_t)$ with $p_1\ge \cdots \ge p_s$ and
	$p'_1\ge\cdots \ge pÕ_t$. Then
	\begin{equation}\label{dominanceeq} 
	P\le P' \text { if  for all } i \text { we have} \sum_{k=1}^i p_k\le\sum_{k=1}^i pÕ_k.
	\end{equation}
Thus, $(2,2,1,1)<(3,2,1)$ but $(3,3,3)$ and $(4,2,2,1)$ are incomparable.\vskip 0.2cm
Given a partition $P=(p_1,\ldots,p_r)$, define its conjugate partition by 
\[P^\vee=(p_1^\vee,\ldots,p_s^\vee), \ \ p_i^\vee=\#\left\{k\mid p_k\geq i\right\}.\]
A partition can be visualized by a Ferrers diagram, and its conjugate partition is then the Ferrers diagram with the rows and columns interchanged.
 
The following Proposition leads to the definition of \emph{generic Jordan type} of $\mathcal{A}$: see \cite[\S 2.2, Lemma 2.8]{IMM1} concerning the semicontinuity of Jordan type. Recall that in a local Artinian algebra $\A$, the maximal ideal $\mathfrak m$ is an affine space. We say that the graded algebra $A$ has \emph{maximal socle degree $j=j_A$} or \emph{formal dimension $j$} (as in \cite{MS}) when $A_j\not=0$ but $ A_i=0$ for $i>j$.\footnote{We follow \cite[Remark 2.11]{H-W} in this definition. When $A$ is not standard graded, we take $\mathfrak m_A=\oplus_{i\ge 1}A_i$; then $j_A$ might not agree with the largest $j$ such that $(\mathfrak m_A)^j\not=0$.}

\begin{proposition}
	\label{propdef:GenJorType}
	Assume that ${\sf k}$ is an infinite field. In the local (resp. graded) case, there exists a non-empty Zariski open dense set $U\subset\mathfrak{m}$ (respectively $U\subset A_1$) such that for each pair of elements $\ell\in U$, $\ell'\in\mathfrak{m}$, (of $\ell \in U, \ell^\prime\in A_1$) we have $P_{\ell'}\leq P_\ell$. 
\end{proposition}
\begin{definition}\label{genericJTdef}  The \emph{generic Jordan type} of $\mathcal{A}$, $P_\mathcal{A}$ (resp. the generic linear Jordan type of $A$, $P_A$) is the Jordan type $P_\ell$ for any element $\ell\in U$, the open dense Zariski set of Proposition~\ref{propdef:GenJorType}
\end{definition}

Recall that the Hilbert function for a local Artinian algebra is the tuple of integers 
\begin{equation}
\label{eq:HFLocal}
H(\mathcal{A})=(h_0,h_1,h_2,\ldots,h_{j_\A}), \ \ h_i=\dim_{\sf k}\left(\mathfrak{m}^i/\mathfrak{m}^{i+1}\right).
\end{equation}

If $\mathcal{A}=A$ is a graded algebra, it comes with its own Hilbert function
\begin{equation}
\label{eq:HFGraded}
H(A)=(h_0,h_1,h_2,\ldots,h_{j}), \ \ h_i=\dim_{\sf k}(A_i).
\end{equation} 

Note that the Hilbert function defines a partition (after possible rearrangement).  We write $H(\mathcal{A})^\vee$ or $H(A)^\vee$ to mean the conjugate partition of the Hilbert function.

For a proof of the following we refer the reader to \cite[Theorem 2.23]{IMM1}.
\begin{proposition}
	\label{prop:HFBound}
	The Jordan type of any element $\ell\in\mathfrak{m}$ is bounded above by the conjugate partition of the Hilbert function of $\mathcal{A}$, i.e. 
	$$P_\ell\leq H(\mathcal{A})^\vee.$$
	If $\mathcal{A}=A$ is graded, and $\ell\in A_1$, then the Jordan type is bounded by the conjugate of the graded Hilbert function, i.e.
	$$P_\ell\leq H(A)^\vee.$$
\end{proposition} 

We say that $\mathcal{A}$ has SLJT (respectively, $A$ is SL) if its generic Jordan type (respectively, generic linear Jordan type) achieves this bound, i.e. 
$$P_{\mathcal{A}}=H(\mathcal{A})^\vee, \ \ \text{resp.} \ \ P_A=H(A)^\vee.$$ 

One can show, e.g. \cite[Proposition 2.37]{IMM1}, that in the graded case, $A$ is strong Lefschetz if and only if there exists a linear form $\ell\in A_1$ for which the multiplication maps $\times\ell^k\colon A_i\rightarrow A_{i+k}$ have full rank for all integers $i,k$.  Moreover, one can further show that if the Hilbert function $H(A)$ is symmetric, i.e. $h_i=h_{j_A-i}$ for each $i$, then SL is in turn equivalent to the condition that the multiplication maps $\times\ell^{j_A-2i}\colon A_i\rightarrow A_{j_A-i}$ are isomorphisms for each $i$.

The following result is well known. For a graded algebra with $A_0\cong \sf k$ we set $\mathfrak m_A=\oplus_{i\ge 1}A_i$. We denote by ${\sf k}\{x,y\}$ the regular local ring in variables $x,y$ over the field $\sf k$.
\begin{lemma}[Height two Artinian algebras are strong Lefschetz]\label{heighttwolem} Let $A={\sf k}[x,y]/I$ be standard Artinian  graded of socle degree $j$, or $\mathcal A={\sf k}\{x,y\}/I$ be local Artinian of socle degree $j=j_\A$,  and suppose $\cha {\sf k}=0$ or $\cha {\sf k}> j$. Let $\ell$ be a general element of $\mathfrak m_A$ in the first case, or of $\maxA$ in the second. Then $\ell$ has SLJT and $A$ is SL (or $\mathcal A$ has SLJT).
\end{lemma}
\begin{proof} These statements follow readily from the openness of the set of directions in which there is a standard basis for ideals in the local ring $\mathbb C\{x,y\}$ \cite[Theorem I.2.1]{Bri}, that extends to the case $\cha {\sf k}=p> j$ (see \cite[Theorem 2.16]{BaI}).
\end{proof}\par

\subsection{Free extensions}\label{groupactionsubsec}
  
\begin{definition}
	\label{def:freeext}\par\noindent
{\bf A.}	Given graded Artinian algebras $A$, $B$, and $C$, we say that $C$ is a free extension of $A$ with fiber $B$ if there exist graded algebra homomorphisms $\iota\colon A\rightarrow C$ and $\pi\colon C\rightarrow B$ for which
	\begin{enumerate}
		\item $\iota$ is injective and makes $C$ into a free $A$-module, 
		\item $\pi$ is surjective and $\ker(\pi)=\iota(\mathfrak{m}_A)\cdot C$.
	\end{enumerate}
Equivalently, $C$ is a free extension of $A$ with fiber $B$ if $\pi$ is surjective and for some (equivalently, every) $\sf k$-linear section of $\pi$, say ${\sf s}\colon B\rightarrow C$, the map $\Phi_{\sf s}\colon A\otimes_{\sf k}B \rightarrow C$ is an isomorphism of $A$-modules.\par\noindent
{\bf B.} Geometrically, as we shall show (Theorem \ref{thm:def}), a free extension $C$ is a flat deformation of a finite (commutative associative) algebra $B$ over a local algebra $(\mathcal A,\mathfrak m,{\sf k})$: so $C$ is an $A$-algebra such that the associated map $\gamma: \Spec(C)\to \Spec(\mathcal A)$ is finite and flat, and with a closed embedding $\pi: \Spec(B) \to \Spec(C)$ inducing a cartesian diagram:\vskip 0.2cm
\qquad\qquad\qquad\qquad\xymatrix{{\Spec(B)}\ar[d]\ar@{^{(}->}^\pi[r]&{\Spec(C)}\ar[d]^\gamma\\
{\Spec ({\sf k})}\ar@{^{(}->}[r]&{\Spec(\mathcal A)}.}

\end{definition}

\begin{remarks}
	\label{rem:FExt}
	\begin{enumerate}
		\item Thus, a free extension $C$ can be regarded as a deformed tensor product of $A$ and $B$, with the deformation occurring in the $B$-factor.  Note that $C$ is isomorphic to the tensor product $A\otimes_{\sf k}B$ as \emph{algebras} if and only if the $\sf k$-linear section ${\sf s}\colon B\rightarrow C$ can be chosen as a map of (graded) $\sf k$-algebras.
		
		\item A related notion is that of coexact sequences which show up in the J.~C.~Moore related topology literature \cite{Bau,MoSm1,Sm1}.  Given a sequence of graded Artinian algebras $A(n)$ $n\in{\mathbb Z}$, we say that a sequence of maps $$\xymatrix{\cdots A({n-1})\ar[r]^-{f_{n-1}} & A(n)\ar[r]^-{f_n} & A({n+1}) \cdots}$$ is \emph{coexact} at $A(n)$ if we have $\ker(f_n)=f_{n-1}(\mathfrak{m}_{A(n-1)})\cdot A(n)$.  The sequence is coexact if it is coexact at $A(n)$ for every $n$. Then $C$ is a free extension of $A$ with fiber $B$ if and only if the sequence 
		\begin{equation}\label{coexactseq}
		\xymatrix{{\sf k}\ar[r] & A\ar[r]^-\iota & C\ar[r]^-\pi & B\ar[r] & {\sf k}}
		\end{equation} is coexact and $\iota\colon A\rightarrow C$ makes $C$ into a free $A$-module.
		
		\item Here is yet another criterion for free extension:  $C$ is a free extension of $A$ with fiber $B$ if and only if $\pi$ is surjective, $\ker(\pi)=\iota(\mathfrak{m}_A)\cdot C$ and $\dim_{\sf k}(C)=\dim_{\sf k}(A)\cdot\dim_{\sf k}(B)$.
	\end{enumerate}
\end{remarks}

\section{Free extension as a deformation, and Jordan type}\label{groupactionsec}
  
We first will show that an $A$-free extension $C$ with fiber $B$ is a deformation of the algebra $A\otimes B$ (Theorem \ref{thm:def}). We will conclude that the generic Jordan type of $C$ is always greater than or equal to that of the actual tensor product $A\otimes B$ (Theorem \ref{SL2thm}).
Let $A$, $B$, and $C$ be graded Artinia $\F$ algebras with homogeneous maximal ideals $\mathfrak{m}_A$, $\mathfrak{m}_B$ and $\mathfrak{m}_C$, respectively.  Suppose that $C$ is a free extension over $A$ with fiber $B$, meaning that there exist algebra homomorphisms $\iota\colon A\rightarrow C$ making $C$ into a free $A$ module and $\pi\colon C\rightarrow B$ surjective with $\ker(\pi)=\mathfrak{m}_AC$ (Definition \ref{def:freeext}). We thank the referee for suggesting we show the following stronger result than we had originally shown and outlining the proof given in Remark~\ref{rem:Rees}  using Rees algebras. We give two different proofs, with the thought that some readers may prefer one. Our first proof was inspired by the discussion in \cite[\S 15.8 and Theorem 15.17]{Ei}, which was pointed out by the referee. 
\begin{theorem}
	\label{thm:def}
	Assume that $\F$ is an infinite field and that $C$ is a free extension over $A$ with fiber $B$. Then there exists a flat $\F[t]$-algebra $\mathfrak{C}$ with fibers $\mathfrak{C}_c\coloneqq \mathfrak{C}/(t-c)\mathfrak{C}$ over $c\in\F$ satisfying 
	\begin{enumerate}
		\item $\mathfrak{C}_0\coloneqq \mathfrak{C}/(t\cdot\mathfrak{C})\cong A\otimes_{\F}B$ as $A$-algebras, and 
		\item $\mathfrak{C}_1\coloneqq \mathfrak{C}/((t-1)\cdot\mathfrak{C})\cong C$ as $A$-algebras, 
		\item for every non-zero $ c\in\F$, we have $\mathfrak{C}_c\cong \mathfrak{C}_1$ as $A$-algebras.
	\end{enumerate}
In other words, $C$ is a flat deformation of $A\otimes_{\F}B$ as an $A$-algebra.
\end{theorem} 
\begin{proof}
	Let $b_1,\ldots,b_r\in B$ be a set of homogeneous elements that generate $B$ as a $\F$-algebra: then $B$ has a homogeneous $\F$-linear basis consisting of some monomials in $b_1,\ldots,b_r$.  Let $c_1,\ldots,c_r\in C$ be any homogeneous $\pi$-lifts of $b_1,\ldots,b_r$, respectively.  Then since $C$ is a free $A$-module with fiber $B$, $C$ must have an $A$-linear basis consisting of some monomials in $c_1,\ldots,c_r$; in particular, the elements $c_1,\ldots,c_r$ must generate $C$ as an $A$-algebra.  Define the surjective $A$-algebra homomorphism 
	$$\phi\colon A[y_1,\ldots,y_r]\rightarrow C, \ \ \phi(y_i)=c_i,$$
	and let $I=\ker(\phi)\subset A[y_1,\ldots,y_r]$.  Since $A$ is Artinian, it is also Noetherian, hence so is $A[y_1,\ldots,y_r]$.  Thus $I$ is finitely generated, say $I=(g_1,\ldots,g_r)$.  Since $\phi$ preserves the grading (if we set $\deg(y_i)=\deg(b_i)$), then we can take the generators $g_1,\ldots,g_r$ to be homogeneous in $A[y_1,\ldots,y_r]$.  Then in $A[y_1,\ldots,y_r,t]$ define the polynomials
	\[g_i^{(t)}\coloneqq t^{\deg(g_i)}\cdot g_i\left(t^{-\deg(y_1)}y_1,\ldots,t^{-\deg(y_r)}y_r\right)\]
	and set 
	$$\mathfrak{C}=\frac{A[y_1,\ldots,y_r,t]}{\bigl(g_1^{(t)},\ldots,g_s^{(t)}\bigr)}.$$ 
	Note that for each $c\in F$ we have 
	$$\mathfrak{C}_c=\mathfrak{C}/(t-c)\mathfrak{C}=\frac{A[y_1,\ldots,y_r,t]}{\bigl(g_1^{(t)},\ldots,g_s^{(t)},t-c\bigr)}\cong\frac{A[y_1,\ldots,y_r]}{\bigl(g_1^{(c)},\ldots,g_s^{(c)}\bigr)}.$$
	In particular, for $c=1$ we have $g_i^{(1)}=g_i$ and hence 
	$$\mathfrak{C}_1\cong \frac{A[y_1,\ldots,y_r]}{(g_1,\ldots,g_s)}\cong C,$$
	which shows that (2) holds.
	
	Note also that for each non-zero $b\in \F$ we have a $\F$-algebra isomorphism 
	$$\xymatrixrowsep{.5pc}\xymatrix{\theta_b\colon A[y_1,\ldots,y_r,t]\ar[r] & A[y_1,\ldots,y_r,t]\\
	a\ar@{|->}[r] & a \\
	y_i\ar@{|->}[r] & b^{\deg(y_i)}y_i\\
	t\ar@{|->}[r] & bt\\}$$ 
	which satisfies $\theta_b(g_i^{(t)})=b^{\deg(g_i)}\cdot g_i^{(t)}$, and $\theta_b(t-1)=b(t-\frac{1}{b})$.  Hence for a non-zero $c\in\F$,  the map $\theta_{c^{-1}}$ passes to an $A$-algebra isomorphism on the quotients
	$$\theta_{c^{-1}}\colon\mathfrak{C}_1= \frac{A[y_1,\ldots,y_r,t]}{\bigl(g_1^{(t)},\ldots,g_s^{(t)},t-1\bigr)}\rightarrow \frac{A[y_1,\ldots,y_r,t]}{\bigl(({c^{-1}})^{\deg(g_1)}g_1^{(t)},\ldots,({c^{-1}})^{\deg(g_s)}g_s^{(t)},{c^{-1}}(t-c)\bigl)}=\mathfrak{C}_{c},$$	
	hence (3) is satisfied.  
	
	To see that (1) holds, note that for each $1\leq i\leq r$ we have 
	$$g_i^{(t)}=g_i^{(0)}+t\cdot h_i$$
	where $g_i^{(0)}\in \F[y_1,\ldots,y_r]$ and $h_i\in A[y_1,\ldots,y_r,t]$ is some polynomial with coefficients in $\mathfrak{m}_A$.  Then we have the following string of $A$-module isomorphisms:
	\begin{align*}
	\mathfrak{C}/(t\cdot \mathfrak{C})= & \frac{A[y_1,\ldots,y_r,t]}{\bigl(g_1^{(t)},\ldots,g_r^{(t)}\bigr)+(t)}\\
	= & \frac{A[y_1,\ldots,y_r]}{\bigl(g_1^{(0)},\ldots,g_r^{(0)}\bigr)}\\
	= & A\otimes_{\F} \frac{\F[y_1,\ldots,y_r]}{\bigl(g_1^{(0)},\ldots,g_r^{(0)}\bigr)}\\
	\cong & A\otimes_{\F} \frac{A[y_1,\ldots,y_r]}{(g_1,\ldots,g_r)+\mathfrak{m}_A}\\
	= & A\otimes_{\F} C/(\mathfrak{m}_A\cdot C)\\
	\cong & A\otimes_{\F} B,
	\end{align*}
	which implies (1).
	 We will now show that $\mathfrak{C}$ is free over $A[t]$, hence flat over $\F[t]$. Since $B\cong\F[y_1,\ldots,y_r]/\bigl(g_1^{(0)},\ldots,g_s^{(0)}\bigr)$ there exist monomials $m_1,\ldots,m_b\in \F[y_1,\ldots,y_r]$ which form a $\F$-linear basis for $B$.  Consequently, the $m_1,\ldots,m_b$ also form an $A$ basis for  $C\cong\mathfrak{C}_1$, and hence also for $\mathfrak{C}_b$ for each non-zero $ b\in F$ (use the isomorphism $\theta_b$ above, and note that $\theta_b(m_i)=m_i$ for each $i$).  We claim that $m_1,\dots,m_b$ form an $A[t]$-module basis for $\mathfrak{C}$.  First, they generate $\mathfrak{C}/\mathfrak{m}_A\mathfrak{C}\cong B\otimes_{\F}\F[t]$ as a $\F[t]$-module, hence by Nakayama's lemma they must also generate $\mathfrak{C}$ as an $A[t]$-module.  Next, suppose by way of contradiction that there is a dependence relation over $A[t]$,
	$$\sum_{i=1}^N\sum_{j=1}^ba_{ij}t^im_j=0, \ \ \text{for some} \ \ a_{ij}\in A.$$
	Then for each $c\in\F$, we must have an $A$-dependence relation in $\mathfrak{C}_c$:
	$$\sum_{j=1}^b\left(\sum_{i=1}^Na_{ij}c^i\right)m_j=0.$$
	This implies that for each $j=1,\ldots,b$ we must have 
	$$\sum_{i=0}^Na_{ij}c^i=0$$
	for every $c\in\F$, and some fixed $a_{ij}\in A$.  But choosing $N$-different points $c_1,\ldots,c_N\in\F$, which is possible since $\F$ is infinite, we can solve the $(N+1)\times(N+1)$ nonsingular Vandermonde system to show that $a_{ij}=0$ for all $i,j$.  Therefore the $m_1,\ldots,m_b$ are $A[t]$-linearly independent, and hence $\mathfrak{C}$ is $A[t]$-free, and in particular $\F[t]$-flat.
\end{proof}\par
\begin{remark}[{\it Proof of Theorem \ref{thm:def} using Rees algebra}]\label{rem:Rees}
We now give a second proof, suggested by the referee, that works over any field $\F$; we use also \cite[\S 6.5]{Ei}.\vskip 0.2cm
Let $C$ be a free extension of $A$ with fibre $B$ as in Theorem \ref{thm:def}.  For each $i$ define the ideal $I_i\coloneqq A_{\geq i}\cdot C$; these ideals define a decreasing filtration $\mathcal{F}(C):I_0=C\supseteq I_1\supseteq I_2\supseteq \cdots$ for which the associated graded algebra is  
$$\operatorname{gr}_{\mathcal{F}(C)}(C)=C/I_1\oplus I_1/I_2\oplus\cdots =\oplus_{i\geq 0}I_i/I_{i+1}.$$

Note that $\operatorname{gr}_{\mathcal{F}(C)}(C)$ is naturally an algebra over the associated graded algebra 
\begin{align*}
\operatorname{gr}_{\mathcal{F}(A)}(A)&=A/A_{\geq 1}\oplus A_{\geq 1}/A_{\geq 2}\oplus\cdots\\
&\cong A_0\oplus A_1\oplus A_2\oplus\cdots \cong A.
\end{align*}
\begin{lemma}
	\label{lem:key} The algebra
	$\operatorname{gr}_{\mathcal{F}(C)}(C)$ is isomorphic to $A\otimes_{\F}B$ as $A$ algebra.
\end{lemma}
\begin{proof}
Note that besides the well defined $\F$-algebra map $\iota\colon A\rightarrow \operatorname{gr}_{\mathcal{F}(C)}(C)$ (which gives $\operatorname{gr}_{\mathcal{F}(C)}(C)$ its $A$-algebra structure), we also have a well defined $\F$-algebra map ${\sf s}\colon B\rightarrow C$ (since $C/I_1=C/\mathfrak{m}_AC\cong B$). Then by the universal property for tensor products, there is also a well defined $A$-algebra homomorphism $$\Phi=\iota\otimes{\sf s}\colon A\otimes_{\F} B\rightarrow \operatorname{gr}_{\mathcal{F}(C)}(C).$$  If $b_1,\ldots,b_r$ is a homogeneous $\F$-basis for $B$, and $a_1,\ldots,a_s$ is a homogeneous basis for $A$, then clearly for each $k$, 
\[\left\{\Phi(a_i\otimes b_j)=a_ib_j+I_{k+1}\mid\deg(a_i)=k\right\}\] 
is a basis for $I_k/I_{k+1}$, hence $\Phi$ is surjective.  Finally since $\dim_{\F}(C)=\dim_{\F}(A)\cdot\dim_{\F}(B)=\dim_{\F}(A\otimes_{\F}B)$, the homomorphism $\Phi$ must be an $A$-algebra isomorphism.
\end{proof}\par
Recall that the \emph{Rees algebra} $\mathcal{R}(C,\mathcal{F})$ for the ring $C$ with filtration $\mathcal{F}=\mathcal{F}(C):I_0\supseteq I_1\supseteq I_2\supseteq\cdots$, is the following $\F[t]$-subalgebra of $C[t,t^{-1}]$ (here we take $I_i=R$ for $i<0$)
$$\mathcal{R}(C,\mathcal{F})=\bigoplus_{i=-\infty}^{\infty}I_{i}\cdot t^{-i}\subset C[t,t^{-1}].$$
\begin{lemma}
	\label{lem:Rees}
	\begin{enumerate}
		\item $\mathcal{R}(C,\mathcal{F})$ is a flat $\F[t]$-algebra.
		\item $R(C,\mathcal{F})/t\cdot\mathcal{R}(C,\mathcal{F})\cong \operatorname{gr}_{\mathcal{F}}(C)$ as $A$-algebras.
		\item $\mathcal{R}(C,\mathcal{F})/(t-c)\cdot \mathcal{R}(C,\mathcal{F})\cong C$ as $A$-algebras, for every $0\neq c\in\F$.
	\end{enumerate}
\end{lemma}
\begin{proof}
	Note that $\mathcal{R}(C,\mathcal{F})\subset C[t,t^{-1}]$ is $\F[t]$-torsion free, hence it is $\F[t]$-flat, which is (1).  To see (2) note that for each integer $i$ there is a natural projection map $\pi_i\colon I_{i}\cdot t^{-i}\rightarrow I_i/I_{i+1}$ whose kernel is $I_{i+1}\cdot t^{-i}\subset I_i\cdot t^{-i}$; taking the direct sum of these maps gives the $A$-algebra map
	$$\oplus_{i\in\Z}\pi_i\colon \mathcal{R}(C,\mathcal{F})\rightarrow \operatorname{gr}_{\mathcal{F}}(C)$$ 
	with kernel 
	$$t\cdot\mathcal{R}(C,\mathcal{F})=\sum_{i=-\infty}^{\infty}I_{i}t^{-i+1}$$
	which gives (2).  Finally note that for every non-zero $c\in\F$, there is also a natural surjective homomorphism
	$$\xymatrixrowsep{.5pc}\xymatrix{\rho\colon\mathcal{R}(C,\mathcal{F})\ar[r] & C\\
		\sum\left(x_i\cdot t^{-i}\right)_{i\in\Z}\ar@{|->}[r] & \sum_{i\in\Z}x_i\cdot c^{-i}.}$$
	Finally note that if $\sum_{i\in\Z}x_ic^{-i}=0$ for some finite sum where each $x_i\in I_i\subset I_{i-1}$, then the difference of $\Z$-tuples 
	$$\left(x_it^{-i}\right)_{i\in\Z}-\left(x_ic^{-1}t^{-i+1}\right)_{i\in\Z}=(t-c)\left(\left(x_it^{-i+1}\right)_{i\in\Z}\right)$$
	is in the kernel and conversely.  Therefore $\ker(\rho)=(t-c)\mathcal{R}(C,\mathcal{F})$, which implies (3).
\end{proof}
\par
		The above argument shows that $C$ is a flat deformation of $A\otimes_{\F}B$ regardless of the field; it also works alike for standard or non-standard graded algebras.  We thank the referee for pointing to the stronger statement, and this proof.  A key observation is Lemma \ref{lem:key}.
	\end{remark}
\begin{theorem}\label{TPthm}
	Let $A$ be an Artinian graded algebra, let $C$ be a free extension of $A$ with fiber $B$, over an infinite field $\ \sf k$.  Let $P_{A\otimes_{\sf k} B}$ be the generic linear Jordan type of the actual tensor product algebra, and let $P_C$ be the generic linear Jordan type of the $A$-free extension $C$.  Then in the dominance partial order we have 
	$$P_C\geq P_{A\otimes_{\sf k} B}.$$
\end{theorem}
\begin{proof}
	Let $\ell=\ell_A\otimes 1+1\otimes\ell_B\in A\otimes_{\sf k} B$ be a linear form whose multiplication map has Jordan type partition $P_{A\otimes_{\sf k} B}$.  By Theorem \ref{thm:def} there is a 1-parameter family $L_t\colon C\rightarrow C$ $(t\in{\sf k}\cong{\mathbb A}^1)$ of degree one endomorphisms of $C$ such $L_0$ has the same Jordan type as the multiplication map $\times\ell\colon A\otimes_{\sf k} B\rightarrow A\otimes_{\sf k} B$, and such that $L_t, t\not=0$ has the same Jordan type as the multiplication map $\times\left(\iota(\ell_A)+t\Lambda\right)\colon C\rightarrow C$.  By the semicontinuity of Jordan type, e.g. Proposition \ref{propdef:GenJorType}, there is an open set $U\in\mathbb A^1$ containing $t=0$ such that the Jordan type $P_{L_t}\geq P_{L_0}$ for all $t\in U$.  Since the generic (linear) Jordan type $P_C$ is the maximal Jordan type occurring for $\ell\in C_1$
	we have, as desired, 
\begin{equation*}
P_C\geq P_{L_t}\geq P_{L_0}=P_{A\otimes_{\sf k} B}.
\end{equation*}  
\end{proof}\par

We can use Theorem \ref{TPthm} to give a new proof of the following result of T. Harima and J.~Watanabe, which they show using central simple modules.\footnote{We have slightly different notation, our $C$ is the $A$ in \cite[Theorem 6.1]{HW}. The definition in \cite{HW} of SL is the ``narrow strong Lefschetz'' of \cite[Definition 3.18]{H-W}, implying that the Hilbert function is symmetric and unimodal. Their statement of Theorem 6.1 in \cite{HW} is for $\cha {\sf k}=0$, but the proof in high enough characteristic $p$ is the same.}
\begin{theorem}\label{SL2thm}\cite[Theorem 6.1]{HW}. Suppose that $C$ is a free extension of $A$ with fiber $B$. Assume that $\cha {\sf k}=0$ or $\cha {\sf k}> j_A+j_B$, that the Hilbert functions of both $A$ and $B$ are symmetric, and that both $A$ and $ B$ are strong Lefschetz. Then $C$ is also strong Lefschetz. 
\end{theorem}
\begin{proof} Under the assumptions on $\cha {\sf k}$,  the Clebsch-Gordan formula applies, and shows that if $A, B$ are strong Lefschetz then the tensor product 
	$A\otimes_{\sf k} B$ is strong Lefschetz.\footnote{T. Harima and J.~Watanabe's \cite[Theorem~3.10]{HW}, also see \cite[Corollary 3.6]{IMM1}.} 
	Hence the generic linear Jordan type is maximal, i.e. $P_{A\otimes B}=H(A\otimes_{\sf k} B)^\vee$.  But since $C$ is a free extension of $A$ with fiber $B$, it must have the same Hilbert function $H(C)=H(A\otimes_{\sf k} B)$, and hence $H(C)^\vee=H(A\otimes_{\sf k}B)^\vee$.   By Proposition \ref{prop:HFBound}, we must have $P_C\leq H(C)^\vee$.  On the other hand Theorem \ref{TPthm} implies that $P_C\geq P_{A\otimes_{\sf k}B}=H(C)^\vee$, and hence we must have equality $P(C)=H(C)^\vee$, and hence $C$ must be SL.
\end{proof}\par

\section{Examples}
\subsection{Free Extensions in Invariant Theory}
Let $V={\sf k}^n$ be a finite dimensional vector space over ${\sf k}$, let $W\subset\Gl(V)$ be any finite group acting linearly on $V$, and let $R=\Sym(V^*)\cong {\sf k}[x_1,\ldots,x_n]$ be the algebra of polynomial functions on $V$ (here $V^*$ is the dual).  Then $W$ acts on $R$ in the usual way, i.e. $(w\cdot f)(v)=f(w^{-1}(v))$, and the set of $W$-invariant polynomials forms a subalgebra $R^W\subset R$ called the $W$-invariant subalgebra.  Let $\mathfrak{h}(W)\subset R$ denote the ideal generated by the non-constant polynomials in $R^W$.  The quotient algebra $R_W=\frac{R}{\mathfrak{h}(W)}$
is called the \emph{coinvariant algebra of $W$}; it is always a graded Artinian algebra.

For any subgroup $K\subset W$, the quotient algebra
$R^K_W=\frac{R^K}{\mathfrak{h}(W)\cap R^K}$
is called the \emph{relative coinvariant algebra for the pair $K\subset W$}.  The inclusion $\hat{\iota}\colon R^K\rightarrow R$ passes to a well defined map of quotient algebras 
$\iota\colon R^K_W\rightarrow R_W$.
We also have an inclusion of ideals $\mathfrak{h}(W)\subset\mathfrak{h}(K)$, and hence a natural surjection of quotient algebras $\pi\colon R_W\rightarrow R_K$.

We say that the invariant subalgebra $R^W$ or $R^K$ is \emph{polynomial} if it can be generated as an algebra by algebraically independent polynomials. 

\begin{lemma}
	\label{lem:RelCoin}
	If $R^K$ is polynomial, then $C=R_W$ is a free extension of $A=R_W^K$ with fiber $B=R_K$ via the maps $\iota$ and $\pi$.  Moreover, if $R^W$ is polynomial and $R_W$ is a free extension of $R^K_W$ with fiber $R_K$ via $\iota$ and $\pi$, then $R^K$ must be polynomial.  
\end{lemma}
\begin{proof}
	Assume first that $R^K$ is polynomial.  We must show that $\iota\colon R_W^K\rightarrow R_W$ makes $R_W$ into a free $A=R^K_W$ module, and that $\ker(\pi)=\iota(\mathfrak{m}_{A})\cdot R_W=(R^K_W)^+\cdot R_W$.  Note that an equivalence class $f+\mathfrak{h}(W)$ is in $\ker(\pi)$ if and only if $f\in \mathfrak{h}(K)$, which immediately implies that $\ker(\pi)=(R^K_W)^+\cdot R_W$.  To see that $R_W$ is a free module, recall, c.f. \cite[Corollary 6.7.13]{Sm1.5}, the fact that $R^K\subset R$ is polynomial if and only if $R$ is a free module over $R^K$.  Therefore $R_W\cong R\otimes_{R^W}{\sf k}$ is a free module over $R^K_W\cong R^K\otimes_{R^W}{\sf k}$.
	
	Assume $R^W$ is polynomial but that $R^K$ is not polynomial.  Consider the natural projection $\hat{\pi}\colon R\rightarrow R_K$.  By Nakayama's Lemma if $\hat{\sf s}\colon R_K\rightarrow R$ is any graded ${\sf k}$-linear section for $\pi$, the map $\hat{\Phi}_{\sf s}\colon R^K\otimes_{\sf k}R_K\rightarrow R$ is surjective; let $\mathcal{K}$ be its kernel so that we have a short exact sequence of $R^K$ modules
	\begin{equation}
	\label{eq:SESInv}
	\xymatrix{0\ar[r] & \mathcal{K}\ar[r] & R^K\otimes_{\F} R_K\ar[r]^-{\Phi_s} & R\ar[r] & 0\\}
	\end{equation}
	Note that since $R^K$ is not polynomial, the fact cited above implies that $\mathcal{K}$ is non-zero.  Applying ${\F}\otimes_{R^W}-$ to the exact sequence \eqref{eq:SESInv}, we get another exact sequence of $\F\otimes_{R^W}R^K\cong R^K_W$-modules
	\begin{equation}
	\label{eq:SES2Inv}\xymatrix{\mathrm{Tor}_1^{R^W}({\F},R)\ar[r] & {\F}\otimes_{R^W} \mathcal{K}\ar[r] & R_W^{K}\otimes_{\F} R_{K}\ar[r]^-{{\Phi}_s} & R_W\ar[r] & 0\\}
	\end{equation}
	
	Since $R^W$ is polynomial, $R$ is a free $R^W$-module hence $\mathrm{Tor}_1^{R^W}({\F},R)=0$.  Moreover since $\mathcal{K}\neq 0$, Nakayama's Lemma implies that ${\F}\otimes_{R^W}\mathcal{K}\cong \mathcal{K}/\mathfrak{h}_W\cap\mathcal{K}\neq 0$ as well.  This implies that the second map ${\Phi}_s=\iota\otimes{\sf s}$ in Sequence \eqref{eq:SES2Inv} is not an isomorphism.  On the other hand, we have seen already that $\ker(\pi\colon R^K_W\rightarrow R_W)=(R^K_W)^+\cdot R_W$ and $R_W/(R^K_W)^+\cdot R_W\cong R_K$, and hence by Nakayama's Lemma $\Phi_{\sf s}\colon R^K_W\otimes_{\F}R_K\rightarrow R_W$ is a free $R^K_W$-module cover of $R_W$ and so is an isomorphism if and only if $R_W$ is free.  We conclude that $R_W$ cannot be a free $R^K_W$-module in this case.  
\end{proof}

\begin{remark}
	\label{rem:Inv}
	In the non-modular case, meaning that $\lvert W\rvert \in\F^\ast$, we have $R^W$ is polynomial if and only if $W$ is generated by pseudo-reflections.  A pseudo-reflection is an invertible linear transformation $s\in\Gl(V)$ with finite order whose fixed point set is a hyperplane in $V$.  Therefore in the non-modular case, Lemma \ref{lem:RelCoin} says that for any fixed group $W$, the coinvariant algebra $R_W$ has a free extension structure for each reflection subgroup $K\subset W$.  A similar result to Lemma \ref{lem:RelCoin} was proved by L. Smith in the non-modular case \cite[Theorem 1]{Sm2}.
\end{remark}

\begin{example}
	\label{ex:SLCoin}
	Let ${\sf k}=\C$, and fix integers $r,n\geq 1$.  Let $W$ be the pseudo-reflection group $G(r,1,n)$, i.e. semi-direct product\small
	$$W=\left\{\left.\left(\begin{array}{ccc} \lambda_1 & \cdots & 0\\ \vdots & \ddots & \vdots\\ 0 & \cdots & \lambda_n\\ \end{array}\right)\right|\lambda_i^r=1\right\}\rtimes\mathfrak{S}_n.$$
	\normalsize
	In other words $W$ is the group of $n\times n$ ``$r$-colored permutation matrices'' whose elements are permutation matrices with non-zero entries arbitrary $r^{th}$ roots of unity.  $W$ acts on the polynomial ring $R={\sf k}[x_1,\ldots,x_n]$ and its invariants are generated by the ``elementary $r$-symmetric polynomials'' in $n$-variables, i.e.
	$$e_i(r,n)= e_i(x_1^r,\ldots,x_n^r)$$
	where $e_i(x_1,\ldots,x_n)$ is the $i^{th}$ elementary symmetric polynomial in the variables $x_1,\ldots,x_n$.  Therefore its coinvariant algebra is the graded Artinian complete intersection 
	$$R_W=\frac{{\sf k}[x_1,\ldots,x_n]}{(e_1(r,n),\ldots,e_n(r,n))}.$$
	As a reflection subgroup take $K=G(r,1,n-1)$, i.e. $K$ is the pointwise stabilizer subgroup of the last coordinate function $x_n$.  Then the $K$-coinvariants are
	$$R_K=\frac{{\sf k}[x_1,\ldots,x_n]}{(e_1(r,n-1),\ldots,e_{n-1}(r,n-1),x_n)}\cong \frac{{\sf k}[x_1,\ldots,x_{n-1}]}{(e_1(r,n-1),\ldots,e_{n-1}(r,n-1))}.$$
	Note that for each $1\leq i\leq n-1$ we have the relations
	$$e_i(r,n)=e_i(r,n-1)+x_n^r\cdot e_{i-1}(r,n-1).$$
	Thus, the $K\subset W$ relative coinvariants are 
	$$R^K_W=\frac{\F[e_1(r,n-1),\ldots,e_n(r,n-1),x_n]}{(e_1(r,n),\ldots,e_n(r,n))}\cong \frac{\F[y_1,\ldots,y_{n-1},x_n]}{(y_1+x_n^r, y_2+x_n^ry_1,\ldots,y_{n-2}+x_n^ry_{n-1},x_n^ry_{n-1})}$$
	or $R^K_W\cong \F[x_n]/(x_n^{nr})$.  We can use Theorem \ref{TPthm} to see that $R_W$ is SL, by induction on $n$.  The base case is $n=2$ and we have 
	$$R_W=\frac{\F[x_1,x_2]}{(x_1^r+x_2^r,x_1^rx_2^r)}$$
	which is SL by Lemma \ref{heighttwolem}.  By induction, $R_K$ is SL, and $R_W^K$ also obviously is SL.  It follows from Corollary \ref{SL2thm} that $R_W$ is SL.
\end{example}

	J. Watanabe et. al. \cite{H-W} used Corollary \ref{SL2thm} in a different way to show that $R_W$ is SL for $W=G(r,1,n)$.  In \cite{McD} Corollary \ref{SL2thm} was used in conjunction with Schubert calculus to show that $R_W$ is SL for every real reflection group.  The next example shows that there are complex reflection groups to which the above argument does not apply.  

\begin{example}
	\label{ex:G333}
	 Let ${\sf k}=\mathbb C$ and take $W$ be the complex reflection group $G(3,3,3)$, i.e.
	 $$W=\left\{\left.\left(\begin{array}{ccc}\lambda_1 & 0 & 0\\ 0 & \lambda_2 & 0\\ 0 & 0 & \lambda_3\\ \end{array}\right)\right|\lambda_i^3=1, \ \lambda_1\lambda_2\lambda_3=1\right\}\rtimes\mathfrak{S}_3.$$
	 In other words $W$ is the group of $3\times 3$ permutation matrices whose non-zero entries are $3^{rd}$ roots of unity with the additional proviso that their product is equal to one.  The coinvariant algebra for $W$ is 
	 $$R_W=\frac{\F[x,y,z]}{(x^3+y^3+z^3,x^3y^3+x^3z^3+y^3z^3,xyz)}.$$
	 Take $K=G(3,3,2)$, i.e. $K$ is the pointwise stabilizer subgroup of the $z$-coordinate.  The coinvariant algebra for $K$ is 
	 $$R_K=\frac{\F[x,y,z]}{(x^3+y^3,xy,z)}\cong \frac{\F[x,y]}{(x^3+y^3,xy)}.$$
	 The relative coinvariant algebra for $K\subset W$ is 
	 $$R_W^K=\frac{\F[x^3+y^3,xy,z]}{(x^3+y^3+z^3,x^3y^3+x^3z^3+y^3z^3,xyz)}\cong \frac{\F[a,b,c]}{(a+c^3,b^3+ac^3,bc)}\cong\frac{\F[b,c]}{(b^3-c^6,bc)}$$
	 where $\deg(a)=3$, $\deg(b)=2$, and $\deg(c)=1$.  Note that the Hilbert function for $R^K_W$ is $H(R^K_W)=(1,1,2,1,2,1,1)$.  The non-unimodality of $H(R^K_W)$ precludes the existence of a SL element.  On the other hand, as a local ring $\mathcal{A}=R^K_W$, and $\ell=b+c\in\mathfrak{m}$ is a SLJT element, the existence of which is guaranteed by Lemma \ref{heighttwolem}. 
\end{example}

In case the Hilbert functions of $A$ and $B$ are symmetric, one can show that both $A$ and $B$ are strong Lefschetz if and only if their tensor product $A\otimes_{\F}B$ is strong Lefschetz.  From this, one can deduce that in Example \ref{ex:G333}, the tensor product $A\otimes_{\F}B$, for $A=R^K_W$ and $B=R_K$, is not SL.  Hence Example \ref{ex:G333} shows that the inequality in Theorem \ref{TPthm} can be strict, i.e. for $C=R_W$, we have $P_C>P_{A\otimes_{\sf k}B}$.  It is tempting to think that this strictness may be an artifact of the non-standard grading on $A$.  The next Example \ref{SL2counterex} shows that a strict inequality $P_C>P_{A\otimes_{\F}B}$ can also occur where each of $A$, $B$, and $C$ has the standard grading, answering a question of J. Watanabe.

\subsection{Free Extensions via Macaulay Duality}

\subsubsection{Dual generator of an Artinian Gorenstein algebra.}
A local Artinian algebra $\mathcal{A}$ is called \emph{Gorenstein} if its socle $(0:\mathfrak{m})$ is a one-dimensional vector space over $\F$.  Let $\mathcal R={\sf k}\{x_1,\ldots ,x_n\}$ be a regular local ring of Krull dimension $n$.  Its \emph{divided power algebra} is an algebra $Q_R={\sf k}_{DP}[X_1,\ldots ,X_n]$ where $X_i^{[s]}\cdot X_i^{[t]}={\binom{s+t}{s}} X_i^{[s+t]}$; $\mathcal R$ acts on $\mathfrak D$ by contraction, i.e. $x_i^k\circ X_j^{[k']}= \delta_{i,j}X_j^{[k'-k]}$ for $k'\ge k$, extended multilinearly. Given an element $F\in Q_R$, define its annihilator ideal $\Ann(F)\subset\mathcal{R}$ as the ideal consisting of all elements of $\mathcal{R}$ which annihilate $F$.  Then the quotient $\mathcal{A}=\mathcal{R}/\Ann(F)$ is a local Artinian Gorenstein algebra.  The module $\hat{\mathcal A}=\mathcal R\circ {F}$ is the Macaulay dual of $\mathcal A$, equivalent to the Macaulay \emph{inverse system} of $\mathcal A$.\footnote{F.H.S. Macaulay used the notation $x_i^{-s}$ for the element $X_i^{[s]}$ in $Q_R$.} 
The \emph{socle} of $\mathcal A$ is $\Soc(\mathcal A)=(0:{\mathfrak m_{\mathcal A}})\subset \mathcal A$, the unique minimal non-zero ideal of $\mathcal A$, and $\dim_{\sf k}\Soc(\mathcal A)=1$. Letting $\mathfrak m_{\mathcal A}$ be the maximal ideal of $\mathcal A$, we have $\Soc (\mathcal A)={\mathfrak m_{\mathcal A}}^{j_\A} \mathcal A\not=0$ and ${\mathfrak m_{\mathcal A}}^{j_\A+1}\mathcal A=0$.
Then we have (\cite[\S 60-63]{Mac}, \cite[Lemma 1.1]{I1}, or, in the graded case, \cite[Lemma 1.1.1]{MS})
\begin{lemma}\label{dualgenlem} i. Assume that $\mathcal A=\mathcal R/I$ is Artinian Gorenstein of socle degree $j$. Then there is a degree-$j$ element ${F}\in Q_R$ such that $I=\Ann {F}$. Furthermore ${F}$ is uniquely determined up to action of a differential unit $u\in \mathcal R$: that is 
	\begin{equation}
	\Ann F=\Ann (u\circ {F}); \text { and } \Ann\  {F}= \Ann\  {G}\Leftrightarrow {G}=u\circ {F} \text { for some unit } u\in \mathcal R.
	\end{equation} The $\mathcal R$-module $(\Ann F)^\perp=\{h\in Q_R$ such that $(\Ann F)\circ h=0\}$ satisfies  $(\Ann F)^\perp =R\circ F$. \par
	ii. Denote by $\phi:\Soc(\mathcal A)\to {\sf k}$ a fixed non-trivial isomorphism, and define the pairing $\langle \cdot,\cdot\rangle_\phi$ on $\mathcal A\times \mathcal A$ by $\langle (a,b)\rangle_\phi=\phi(ab)$.  Then the pairing $\langle (\cdot,\cdot)\rangle_\phi$ is an exact pairing on $ \mathcal A$, for which $({{\mathfrak m}}^i)^\perp=(0:\mathfrak {m_{\mathcal A}}^i)$.  We have $0:{\mathfrak m_{\mathcal A}}^i=\Ann ({\mathfrak m_{\mathcal A}}^i\circ {F})$.
	Also $\Ann (\ell^i\circ {F})=I:\ell^i$.
\end{lemma}
When $\mathcal A=\mathcal R/ I$ is a local ring then in general $F$ is not homogeneous.  When $\mathcal{A}$ is (possibly non-standard) graded we will write $A$ for $\mathcal A$: then the dual generator ${F}\in Q_R$ may be taken homogeneous, and it is unique up to non-zero scalar multiple.

{J. Watanabe asked whether there is a converse to Theorem \ref{SL2thm} if we assume that each of $A,B,C$ are standard graded.\footnote{Recall that T. Harima and J. Watanabe gave a counterexample to the converse of Theorem \ref{SL2thm} when $A$ is allowed to have non-standard grading \cite[Example 6.3]{HW2}.}\vskip 0.2cm\noindent
\begin{question}\label{junzoquest} Assume that the $A$-free extension $C$ with fiber $B$ is strong Lefschetz, and that $A,B,C$ are standard graded. Can we conclude that $A,B$ are SL?
\end{question}
We will show that the answer to the Question \ref{junzoquest} is ``No'' in Example \ref{SL2counterex}. In order to show this example we prove a result about freeness of extensions $C$ over the ring $A={\sf k}[t]/(t^{m+1})$.}\par
Let $R={\F}[x_1,\ldots,x_n]$ be a standard graded polynomial ring, let
$I_B\subset R$ be a homogeneous ideal of finite colength such that the quotient $B=R/I_B$ is a graded Artinian Gorenstein algebra with socle degree $j_B$. Let $Q_R={\F}[X_1,\ldots,X_n]$ be the divided power algebra corresponding to $R$, and let $F_B\in (Q_R)_{j_B}$ be a (homogeneous) Macaulay dual generator for $B$ as in Lemma~\ref{dualgenlem}. Set
$A={\F}[t]/(t^{m+1})$, and consider $F_A=T^{m}$, a Macaulay dual generator for $A$ in ${\F}[T]$. Let $S={\F}[x_1,\ldots,x_n,t]=R[t]$ and set $Q_S={\F}[X_1,\ldots,X_n,T]$, the corresponding ring of divided powers.
\begin{lemma}\label{specialdualgen-lem}
	Let $G\in Q_R$ be a homogeneous polynomial of degree $\deg G=j_B+m$, and consider the polynomial in $F\in Q_S$ defined by
	\[
	F=T^{[m]}\cdot F_B+G.
	\]
	Then $F$ is the Macaulay dual generator of a free extension with base $A$ and fiber $B$ if and only if $(I_B)^2\circ G = 0$.
\end{lemma}
\begin{proof} We first show $\Leftarrow$, that $(I_B)^2\circ G=0$ implies $C$ is an $A$-free extension.
	Let ${C=S/I_C}$, where ${I_C=\Ann_SF}$ and assume  $I_B^2\circ G=0$. Consider the projection map ${\hat{\pi}:S\rightarrow R}$ defined by ${x_i\mapsto x_i}$ and
	${t\mapsto 0}$. Let ${f\in I_C}$, where 
	\begin{equation}\label{feqn}
	{f=\sum_0^k t^if_i=f_0+tf_1+\cdots+t^kf_k},\, \text { with } {f_i\in R}.
	\end{equation}
	We assume ${k\ge m}$, as we do not require ${f_k\ne0}$. Since \[
	0 = f\circ F = T^{[m]}f_0\circ F_B + T^{[m-1]}f_1\circ F_B + \cdots +
	Tf_{m-1}\circ F_B
	+f_m\circ F_B + f_0\circ G, \]
	we must have ${f_0\circ F_B = 0}$, so ${\hat{\pi}(f)=f_0\in I_B}$. This shows that $\hat{\pi}$ yields a morphism ${\pi:C\to B}$, which is surjective, by construction.
	Also, the natural inclusion $\hat{\iota}\colon {\F}[t]\rightarrow \hat{R}$ passes to a map $\iota\colon A\rightarrow C$, since ${t^{m+1}\circ F=0}$.\par
	We now wish to show that ${\ker\pi=(\iota(A)_+)C}$. The inclusion ${(\iota(A)_+)C\subseteq\ker\pi}$ is immediate. For the other inclusion, consider ${{\sf f}\in\ker\pi}$, with ${{\sf f}={\sf f}_0+t{\sf f}_1+\cdots+t^m{\sf f}_m}$, ${{\sf f}_i\in R}$ (we do not need higher degree in $t$, since ${t^{m+1}\in I_C}$). We have $0=\pi({\sf f})=\overline{{\sf f}}_0\in B$, which implies ${\sf f}_0\in I_B$. Since ${(I_B)^2\circ G = 0}$, we have ${I_B\circ G\subseteq R\circ F_B}$, so there exists ${g_m\in R}$ such that ${{\sf f}_0\circ G = g_m\circ F_B}$. Let ${g=t{\sf f}_1+\cdots+t^m({\sf f}_m+g_m)\in(t)\subset S}$. Then ${{\sf f}-g={\sf f}_0-t^mg_m\in I_C}$ as ${\sf{f}}_0\in I_B$, and
	$$({\sf f}-g)\circ F=({\sf f}_0-t^mg_m)\circ  (T^{[m]}F_B+G)=T^{[m]}({\sf f}_0\circ F_B)+{\sf f}_0\circ G-g_m\circ F_B=0.$$
	Hence, as an element of $C$, we have ${\sf f}\in \ker\pi$ implies ${{\sf f}\in (\iota(A)_+)C}$, hence the sequence  $\xymatrix{{\sf k}\ar[r] & A\ar[r]^-\iota & C\ar[r]^-\pi & B\ar[r] & {\sf k}}$ is coexact at $C$; the coexactness at $A$ and $B$ are obvious. \par	
	To complete the proof of $\Leftarrow$ we need to show \vskip 0.2cm\noindent
	{\bf Claim.} We have equality $\dim_{\sf k}C=\dim_{\sf k}(A)\cdot \dim_{\sf k}(B)=(m+1)\cdot\dim_{\sf k}(B)$.
	 This implies that $C$ is a free $A$ module via the inclusion $\iota\colon A\rightarrow C$. \vskip 0.2cm\noindent
	{\it Proof of Claim.} Since $C\supset tC\supset t^2C \cdots  \supset t^mC\subset t^{m+1}C=0$ is a filtration of $C$, we have that as vector space $C\cong_{\sf k}\oplus_{i=0}^m t^iC/(t^{i+1}C)$. We have shown above that $\pi$ is surjective with $\ker\pi=tC$, and that  $ {{\sf f}={\sf f}_0+t{\sf f}_1+\cdots+t^m{\sf f}_m}$ satisfies ${\sf f}\in\ker\pi$ implies ${\sf f}_0\in I_B$. Thus we have $C/\ker\pi\cong C/tC\cong B$. Assume by way of induction that for an integer $i\in [0,m-1]$ we have $t^iC/t^{i+1}C\cong B$, and consider the homomorphism $m_t$ induced by the multiplication $ c\to  t\cdot c$
	\begin{equation}\label{inducteq} 
	m_t:  t^iC/t^{i+1}C\to t^{i+1}C/t^{i+2}C.
	\end{equation}
	Evidently $m_t$ is surjective. We now show $m_t$ is injective. Suppose that for an $\overline{\alpha} \in t^iC/(t^{i+1}C)$ we have $m_t(\overline{\alpha})=0$; let $\alpha=t^i{{\sf c}_0}\in t^iC$ be a representative of $\overline{\alpha}.$  
	Then there exists ${\sf c}_1\in C$ for which $t^{i+1}{\sf c}=t^{i+2}{\sf c}_1$; in particular then we must have $t^m{\sf c}_0=0$, which implies that $0=(t^m{{\sf c}_0})\circ F={\sf c}_0\circ F_B$, hence ${\sf c}_0\in I_B$, implying $t^i{\sf c}_0\in t^{i+1}C$ implying $\overline{\alpha}=0$. Hence $m_t$ is an isomorphism. This completes the induction step. We have shown equation \eqref{inducteq} for $i=0$, thus, by induction we have $\dim_{\sf k}t^iC/t^{i+1}C=\dim_{\sf k}B$ for $0\le i\le m$, implying that $\dim_{\sf k} C=(m+1)\dim_{\sf k}(B)$. This completes the proof of the claim,
	and the proof that $(I_B)^2\circ G = 0$ implies that $C$ is a free extension of $A$ with fiber $B$.\par
	
	We now prove $\Rightarrow$. Assume that $C$ is an $A$-free extension. Then ${\ker\pi=(\iota(A)_+)C}$. Here ${\sf f}=\sum_0^k t^i{\sf f}_i\in \ker \pi \Leftrightarrow {\sf f}_0\in I_B$ so $I_B\subset tC$. So there is $h\in C$ so $({\sf f}_0-th)\circ F=0$, but $({\sf f}_0-th)\circ F=-th\circ (T^{[m]}F_B)+{\sf f}_0\circ G$, so we have the highest degree term $t^{m-1}h_0$ of $h$ satisfies $h_0\circ F_B={\sf f}_0\circ G$, so $I_B\circ ({\sf f}_0\circ G)=I_B\circ (h_0\circ F_B)=0$. Since this occurs for all ${\sf f}_0\in I_B$ we have  $(I_B)^2\circ G=0$. This completes the proof of the Theorem.\footnote{This result is generalized in \cite[Theorem 2.1]{IMM2}.}
\end{proof}\par
\begin{remark} Lemma \ref{specialdualgen-lem} might also be derived from results appearing in the papers of J.~Elias and M. E. Rossi \cite{ElRo} or J. Jelisiejew \cite{Je}.  For example in terms of the paper \cite{ElRo}, the sequence $H_i=T^{[i]}F_B$ for $0\leq i\leq m-1$ and $H_m=T^{[m]}F_B+G$ forms a ``truncated'' $G_1$-admissible system generating an $S$-submodule of $Q_S$ which one could then show corresponds to a flat extension over $A=\F[t]/(t^{m+1})$ using an argument similar to the proof of their \cite[Theorem 3.8]{ElRo}.  We thank the referee for pointing out this connection. 
	
\end{remark}
We below will use the notation $m\otimes n$ to denote the generic Jordan type $(m+n-1,m+n-3,\ldots )$ of ${\sf k}[x,y]/(x^m,y^n)$.

\begin{example}[$C$ is an $A$-free extension, $C$ is SL, but $B$ is not SL]
\label{SL2counterex}
Let $\sf k$ be an infinite field of characteristic zero or
characteristic $p\ge 5$. Take $R={\sf k}[x,y,z,u,v]$, $S=R[t]$, let
$A={\sf k}[t]/(t^2)$ of Hilbert function $H(A)=(1,1)$ and let $B=R
/I_B$, $I_B=\Ann F_B$, $F_B=(XU^{[2]}+YUV+ZV^{[2]})$, an idealization of
${\sf k}[u,v]/(u,v)^2$. Then $H(B)=(1,5,5,1)$,
$I_B=\bigl((x,y,z)^2,uy-vz,ux-vy,uz,vx\bigr)$, and it is
straightforward to see that the Jordan type of $B$ is
$J_B=(4,2,2,2,1,1)$, so $B$ is not strong Lefschetz, and, since it has
symmetric Hilbert function, also cannot have an element of strong
Lefschetz Jordan type.\footnote{The cubic defining $B$ was studied by
U.~Perazzo in 1900, see \cite[Example 7.5.1, Theorem 7.6.8]{Ru}.} Let
$C=S/I$, $I=\Ann F$ where $F=TF_B+G$, $G=X^{[2]}UV + XYV^{[2]}$. It is
straightforward to verify that $(I_B)^2\circ G=0$, so $C$ is an $A$-free
extension with fiber $B$ by Lemma \ref{specialdualgen-lem}, and we have
$H(C)=(1,6,10,6,1)$.  A calculation in Macaulay2 shows that the (usual)
Hessian of $F$ is non-zero, hence for a generic linear form $\ell\in S$
we have the multiplication $\ell^2: C_1\to C_3$ is an isomorphism.
Evidently, for a generic $\ell$,  we have $\ell^4\not=0$.\footnote{in general, for any standard graded Artinian algebra $S$ of socle degree $j$ and any sufficiently general linear form $\ell$ we always have $\ell^j\neq 0$ in $S$, as the powers
of linear forms span $S_j$ when $\cha {\sf k}=0$ or is greater than
$j$.} Therefore, $P_C=(5,3^5,1^4)$, so $C$ is strong Lefschetz.  As an example, we can check that for $\ell=x+u+v$, we have
\begin{align*}
\ell^2x\circ F&=2X+2U+2V &\ell^2y\circ F&=2V+2T\\
\ell^2z\circ F&=T &\ell^2u\circ F&=2X+V+2T\\
\ell^2v\circ F&=2X+2Y+U &\ell^2t\circ F&=X+2Y+Z+2U
\end{align*}
and since these are linearly independent in the dual space, so are
$\ell^2x,\,\ell^2y,\,\ell^2z,\,\ell^2u,\,\ell^2v,\,\ell^2t$ linearly independent in $C_3$.
Therefore $\times\ell^2:C_1\to C_3$ is an isomorphism. Furthermore,
$\ell^4=12x^2uv\ne0$:  these verify that $\ell$ is a SL element for $C$.\par
Note that here the Jordan type of $A\otimes B$ is $2\otimes
(4+2^3+1^2)=2\otimes 4+3(2\otimes 2)+2(2\otimes 1)=(5,3)\cup 3(3,1)\cup
(2,2)=(5,3^4,2^2,1^3)$, thus $P_C>P_{A\otimes B}$ and $P_C$ covers
$P_{A\otimes B}$ in the dominance partial order.
\end{example}
\begin{ack}
	The first and third author appreciate conversations with participants of the informal work group on Jordan type, and others at the workshop ``Lefschetz Properties in Geometry, Algebra, and Combinatorics'' at Institute Mittag Leffler in July, 2017. We appreciate comments and answers to our questions by Larry Smith, Junzo Watanabe and Alexandra Seceleanu. We thank Oana Veliche, Ivan Martino, Jerzy Weyman, Emre Sen, and Shujian Chen for their comments. We thank the referee for cogent and helpful comments including an improved statement and proof for Theorem \ref{thm:def}. \par
The second author was partially supported by CIMA -- Centro de Investiga\c{c}\~{a}o em Mate\-m\'{a}tica e Aplica\c{c}\~{o}es, Universidade de \'{E}vora, project UID/MAT/04674/2019 (Funda\c{c}\~{a}o para a Ci\^{e}ncia e Tecnologia). 
\end{ack}

\end{document}